\documentclass[12pt, reqno]{amsart}
\usepackage{graphicx}
\usepackage{hyperref}
\usepackage[caption = false]{subfig}
\usepackage{amsthm}
\usepackage{amssymb}
\usepackage{mathrsfs}
\usepackage{mathtools}
\usepackage{enumerate}
\usepackage{amssymb}
\usepackage{wrapfig}
\usepackage{float}
\allowdisplaybreaks
 
\usepackage{subfig}
\usepackage[twoside,paperwidth=190mm, paperheight=297mm, top=35mm, bottom=20mm, left=20mm, right=20mm, marginparsep=3mm, marginparwidth=40mm]{geometry}

\numberwithin{equation}{section}

\theoremstyle{plain}

\newtheorem{theorem}{Theorem}[section]
\newtheorem{corollary}[theorem]{Corollary}
\newtheorem{example}[theorem]{Example}
\newtheorem{lemma}{Lemma}[section]

\newtheorem{conj}[theorem]{Conjecture}
\theoremstyle{definition}
\newtheorem{definition}[theorem]{Definition}
\theoremstyle{remark}
\newtheorem{remark}{Remark}[section]

\makeatother

\begin{document}

\title{On Certain Generalizations of $\mathcal{S}^*(\psi)$}
	\thanks{The work of the second author is supported by University Grant Commission, New-Delhi, India  under UGC-Ref. No.:1051/(CSIR-UGC NET JUNE 2017).}	
	
	\author[S. Sivaprasad Kumar]{S. Sivaprasad Kumar}
	\address{Department of Applied Mathematics, Delhi Technological University,
		Delhi--110042, India}
	\email{spkumar@dce.ac.in}

	\author[Kamaljeet]{Kamaljeet Gangania}
	\address{Department of Applied Mathematics, Delhi Technological University,
		Delhi--110042, India}
	\email{gangania.m1991@gmail.com}

\maketitle	
	
\begin{abstract}
	We deal with different kinds of generalizations of $\mathcal{S}^*(\psi)$, the class of Ma-Minda starlike functions, in addition to a sharp majorization result of $\mathcal{C}(\psi),$ the class of  Ma-Minda convex functions, in terms of radius problems.  We also obtain a sufficient condition for the functions to be in $\mathcal{S}^*(\psi)$.  For a fixed $f\in \mathcal{S}^*(\psi),$ the  class of subordinants $S_{f}(\psi):= \{g : g\prec f  \} $ is introduced and studied for the Bohr-phenomenon and a couple of conjectures are also proposed.
\end{abstract}
\vspace{0.5cm}

\maketitle
	
	\section{Introduction}
	\label{intro}
	Let $\mathcal{A}$ denote the class of analytic functions of the form $f(z)=z+\sum_{k=2}^{\infty}a_kz^k$ in the open unit disk $\mathbb{D}:=\{z: |z|<1\}$. Using subordination,	Ma-Minda \cite{minda94} introduced the unified class of starlike and convex functions defined as follows:
	\begin{equation}\label{mindaclass}
	\mathcal{S}^*(\psi):= \biggl\{f\in \mathcal{A} : \frac{zf'(z)}{f(z)} \prec \psi(z) \biggl\} \; \text{and} \;  \mathcal{C}(\psi):= \biggl\{f\in \mathcal{A} : 1+\frac{zf''(z)}{f'(z)} \prec \psi(z) \biggl\},
	\end{equation}
	where  $\psi$ is a Ma-Minda function, which is   analytic and univalent with $\Re{\psi(z)}>0$, $\psi'(0)>0$, $\psi(0)=1$ and $\psi(\mathbb{D})$ is symmetric about real axis. Note that $\psi \in \mathcal{P}$, the class of normalized Carath\'{e}odory functions. The class $\mathcal{S}^*(\psi)$  unifies various subclasses of starlike functions, which are obtained for an appropriate choice of $\psi$. Ma-Minda discussed many properties of the class $\mathcal{S}^*(\psi)$, in particular, they proved the distortion theorem \cite[Theorem~2, p.162]{minda94} with some restriction on $\psi$, namely
	\begin{equation}\label{minda-cond}
	\underset{|z|=r}{\min}|\psi(z)|=\psi(-r)\quad \text{and}\quad \underset{|z|=r}{\max}|\psi(z)|=\psi(r).
	\end{equation}
	In section~\ref{sec-3}, we modify the distortion theorem by relaxing this restriction on $\psi$ to obtain a more general result.
	In 1914, Harald Bohr \cite{bohr1914} proved the following remarkable result related to the power series:
	\begin{theorem}[\cite{bohr1914}]
		Let $g(z)=\sum_{k=0}^{\infty}a_kz^k$ be an analytic function in $\mathbb{D}$ and $|g(z)|<1$ for all $z\in \mathbb{D}$, then
		$\sum_{k=0}^{\infty}|a_k|r^k\leq1$
		for all $z\in \mathbb{D}$ with $|z|=r\leq1/3$.
	\end{theorem}
	Bohr actually proved the above result for $r\leq1/6$. Further Wiener, Riesz and Shur independently sharpened the result for $r\leq1/3$. Presently the Bohr inequality for functions mapping unit disk onto different domains, other than unit disk is an active area of research. For the recent development on Bohr-phenomenon, see the articles \cite{ali2017,jain2019,bhowmik2018,boas1997,ponnusamy17,ponnusamy18,muhanna10,muhanna11,muhanna14,muhana2016,singh2002} and references therein.
	The concept of Bohr-phenomenon in terms of subordination can be described as:
	Let $f(z)=\sum_{k=0}^{\infty}a_kz^k$ and $g(z)=\sum_{k=0}^{\infty}b_kz^k$ are analytic in $\mathbb{D}$ and $f(\mathbb{D})=\Omega$. For a fixed $f$, consider a class of analytic functions $S(f):=\{g : g\prec f\}$ or equivalently $S(\Omega):=\{g : g(z)\in \Omega\}$. Then the class $S(f)$ is said to satisfy Bohr-phenomenon, if there exists a constant $r_0\in (0,1]$ satisfying the inequality $
	\sum_{k=1}^{\infty}|b_k|r^k \leq d(f(0),\partial\Omega)$
	for all $|z|=r\leq r_0$ and $g(z) \in S(f)$, where $d(f(0),\partial\Omega)$ denotes the Euclidean distance between $f(0)$ and the boundary of $\Omega=f(\mathbb{D})$. The largest such $r_0$ for which the inequality holds, is called the Bohr-radius.
	
	In  2014, Muhanna et al. \cite{muhanna14} proved  the Bohr-phenomenon for $ S(W_{\alpha})$, where  $W_{\alpha}:=\{w\in\mathbb{C} : |\arg{w}|<\alpha\pi/2, 1\leq \alpha \leq2\},$  which  is a Concave-wedge domain (or exterior of a compact convex set) and the class $R(\alpha,\beta, h)$ defined by $R(\alpha,\beta,h) := \{g\in \mathcal{A} : g(z)+\alpha z g'(z)+\beta z^2 g''(z)\prec h(z)\},$ where $h$ is a convex function (or starlike) and $R(\alpha,\beta, h)\subset S(h)$.	In 2018, Bhowmik and Das \cite{bhowmik2018} proved the Bohr-phenomenon for the classes given by
	$S(f)=\{g\in\mathcal{A} : g\prec f\; \text{and}\; f\in \mu(\lambda) \}$, where  $\mu(\lambda)=\{f\in\mathcal{A} : |(z/f(z))^2f'(z)-1|<\lambda, 0<\lambda\leq1\}$ and $S(f)=\{g\in \mathcal{A}: g\prec f \;\text{and}\; f\in \mathcal{S}^*(\alpha), 0\leq\alpha\leq1/2\}$,
	where $\mathcal{S}^*(\alpha)$ is the well-known class of starlike functions of order $\alpha$.	In Section~\ref{sec-4}, for any fixed  $f\in \mathcal{S}^*(\psi)$,  we introduce and study the Bohr-phenomenon inside the disk $|z|\leq1/3$ for the following class of analytic subordinants:
	\begin{equation}\label{bohrclass}
	S_{f}(\psi):= \biggl\{g(z)=\sum_{k=1}^{\infty}b_k z^k : g \prec f \biggl \}.
	\end{equation}
	Note that $\mathcal{S}^*(\psi)\subset \bigcup_{f\in \mathcal{S}^*(\psi)} S_{f}(\psi)$. As an application, we obtain the Bohr-radius for the class $S(f)$, where $f\in \mathcal{S}^*({(1+Dz)}/{(1+Ez)})$, the class of Janowski starlike functions, with some additional restriction on $D$ and $E$ apart from $-1\leq E< D\leq1$.
	Now recall the following definition and a result due to T. H. MacGregor~\cite{mc}:
	\begin{definition} (\cite{mc})\label{mc-def}
			Let $f$ and $g$ be analytic in $\mathbb{D}$. A function  $g(z)$ is said to be majorized by $f(z)$, denoted by $g<<f$, if there exists an analytic function $\Phi(z)$ in $\mathbb{D}$ satisfying
			$ |\Phi(z)|\leq1$ and $g(z)=\Phi(z) f(z)$ for all $z\in\mathbb{D}.$
	\end{definition}		
	
	\begin{theorem}[\cite{mc}]
		Let $g$ be majorized by $f$ in $\mathbb{D}$ and $g(0)=0$. If $f(z)$ is univalent in $\mathbb{D}$, then $|g'(z)|\leq|f'(z)|$ in $|z|\leq 2-\sqrt{3}$. The constant $2-\sqrt{3}$ is sharp.
	\end{theorem}
	
	Recently Tang and Deng \cite{tang} obtained the majorization results for $ \mathcal{S}^*(\psi)$ for some specific choices of $\psi$, motivated by this in section~\ref{sec-1}, we devise a  general approach to handle the same for $ \mathcal{C}(\psi)$, which is precisely stated as: if $g\in\mathcal{A}$, $f\in \mathcal{C}(\psi)$ and  $g$ is majorized by $f$ in $\mathbb{D}$, then we
	find the largest radius $r_{\psi}\leq1$ such that $|g'(z)|\leq |f'(z)|$ in $|z|\leq r_{\psi}$. Several other results in this direction are also obtained. In section~\ref{sec-2}, we consider the radius problem posed by Obradovi\'{c} and Ponnusamy \cite{obPonnu} namely: Let $g \in \mathcal{S}^*(\psi_1)$ and $h \in \mathcal{S}^*(\psi_2)$, then find the largest radius $r_0\leq1$ such that the function $F(z)={(g(z)h(z))}/{z}$ belongs to certain well-known class of starlike functions in $|z|<r_0$. As a special case, we also obtain a result of Obradovi\'{c} and Ponnusamy \cite{obPonnu}. Further we obtain the condition for functions to be in $\mathcal{S}^*(\psi)$ which is an extention of the Bulboac\u{a} and Tuneski~\cite{Bulboca-2003}. Throughout this paper we shall assume that the function $\phi$ has real coefficients in it's power series expansion.
	
	\section{Majorization}\label{sec-1}
Let us consider the analytic function
$\psi(z):=1+B_1z+B_2z^2+\cdots.$
Here $B_1=\psi'(0)$, the coefficient of $z$, plays a major role in deciding the orientation of the function $\psi$. Thus $\psi$ is positively or negatively oriented depends on whether $B_1$ is positive or negative. Ma-Minda only considered the case $\psi'(0)>0$, as it may be possible that for the case when $\psi'(0)<0$, many postulates for the class $\mathcal{S}^{*}(\psi)$ need not remain same. With this perspective, we begin with the following:

\begin{theorem}\label{mj}
	Let $\Re\phi(z)>0$ and $\phi$ be convex in $\mathbb{D}$ with $\phi(0)=1$.
	Suppose $\psi$ be the function such that $m_r:=\underset{|z|=r}{\min}|\psi(z)|$
	and also satisfies the differential equation
	\begin{equation}\label{briot-sol}
	\psi(z)+\frac{z\psi'(z)}{\psi(z)}=\phi(z).
	\end{equation}
	Let $g\in \mathcal{A}$ and $f\in \mathcal{C}(\phi)$. If $g$ is majorized by $f$ in $\mathbb{D}$, then
	\begin{equation}\label{m}
	|g'(z)| \leq |f'(z)| \quad \text{in}\;|z|\leq r_{\psi},
	\end{equation}
	where $r_{\psi}$ is the least positive root of the equation
	\begin{equation}\label{r0}
	(1-r^2)m_r-2r=0.
	\end{equation}
	The result is sharp for $m_r=\psi(-r)$.
\end{theorem}
\begin{proof}
	Let us define $p(z):=zf'(z)/f(z)$. Since  $f\in \mathcal{C}(\phi)$, therefore we have $1+zf''(z)/f'(z)\prec \phi(z)$, which can be equivalently written as
	\begin{equation}\label{briot}
	p(z)+\frac{zp'(z)}{p(z)}=1+\frac{zf''(z)}{f'(z)} \prec \phi(z).
	\end{equation}
	Since $\Re\phi(z)>0$ and $\phi$ is convex in $\mathbb{D}$, therefore using \cite[Theorem~3.2d, p.~86]{subbook} the solution $\psi$ of the differential equation \eqref{briot-sol}
	is analytic in $\mathbb{D}$ with $\Re\psi(z)>0$ and has the following integral form given by
	\begin{equation*}
	\psi(z):= h(z)\left(\int_{0}^{z}\frac{h(t)}{t}dt\right)^{-1},
	\end{equation*}
	where
	\begin{equation*}
	h(z)= z\exp\int_{0}^{z}\frac{\phi(t)-1}{t}dt.
	\end{equation*}
	Since $\Re\psi(z)>0$ and $p$ satisfies the subordination~\eqref{briot}, therefore using \cite[Lemma~3.2e, p.~89]{subbook} we conclude that $\psi$ is univalent and $p\prec \psi$, where $\psi$ is the best dominant.
	Thus we have obtained that $f\in\mathcal{C}(\phi)$ implies $zf'(z)/f(z) \prec \psi(z)$ and $\psi$ is the best dominant (which is important for the sharpness of result), which is a univalent Carathe\'{o}dory function.
	Now as  $g\in \mathcal{A}$ and  $f\in \mathcal{C}(\phi)$, therefore we obtain the following well defined equality
	$$\frac{f(z)}{f'(z)}=\frac{z}{\psi(\omega(z))}, \quad (z\in \mathbb{D})$$
	where $\omega$ is a Schwarz function. Hence, using
	$\min_{|z|=r}|\psi(\omega(z))|\geq \min_{|z|=r}|\psi(z)|$
	and the hypothesis $\underset{|z|=r}{\min}|\psi(z)|=m_r$, we obtain that
	\begin{equation}\label{m1}
	\left|\frac{f(z)}{f'(z)}\right|\leq \frac{r}{m_r}, \quad (0<r<1).
	\end{equation}
	Now if $g$ is majorized by $f$, then by definition, we have $g(z)=\Psi(z) f(z)$, where $\Psi$ is analytic and satisfies $|\Psi(z)|\leq1$ in $\mathbb{D}$ such that $g'(z)=\Psi(z)f'(z)+{\Psi}'(z)f(z)$. Thus using \eqref{m1} together with the following Schwarz-Pick inequality
	$$|\Psi'(z)| \leq \frac{1-|\Psi(z)|^2}{1-|z|^2},$$
	we obtain
	\begin{equation}\label{unique}
	|g'(z)| \leq |f'(z)|\left(|\Psi(z)|+ \frac{1-|\Psi(z)|^2}{1-r^2}\frac{r}{m_r} \right) = |f'(z)|h(\beta,r),
	\end{equation}
	where $|\Psi(z)|:=\beta$ and
	$$h(\beta,r)=\beta+\frac{1-\beta^2}{1-r^2}\frac{r}{m_r}.$$
	Thus to arrive at \eqref{m}, it suffices to show that $h(\beta,r)\leq1$, which is equivalent to show that
	\begin{equation}\label{m2}
	k(\beta,r):= (1-r^2)m_r -(\beta+1)r\geq 0.
	\end{equation}
	Since $\frac{\partial}{\partial\beta}k(\beta,r)=-r<0$, Therefore, \eqref{m2} holds whenever
	$$ k(r):=\min_{\beta}k(\beta,r)= k(1,r)\geq0.$$
	Note that $k(r)$ is a continuous function of $r$ and further $k(0)= m_{0}=\psi(0)=1>0$ and $k(1)<0$. Thus there exists a point $r_{\psi}\in (0,1)$ such that $k(r)\geq0$ for all $r\in [0,r_{\psi}]$, where $r_{\psi}$ is the least positive root of \eqref{r0}.
	
	{\bf{For sharpness:}} Now let $m_r=\psi(-r)$. Choose $f(z)\in \mathcal{C}(\phi)$ such that $zf'(z)/f(z)=\psi(-z)$ and $\Psi(z)=(z+\alpha)/(1+\alpha z)$, where $-1\leq \alpha \leq 1$. We show that for each $r_{\psi} < r \leq 1$, we can choose $\alpha$ so that $g'(r)>f'(r)>0$, which implies that $g'$ is not majorized by $f'$ outside $|z|\leq r_{\psi}$. First note that
	\begin{equation}\label{m3}
	\frac{f(r)}{f'(r)}=\frac{r}{\psi(-r)}.
	\end{equation}
	Since
	\begin{equation*}
	g'(r)= f'(r) \left( \frac{r+\alpha}{1+\alpha r}+ \frac{1-{\alpha}^2}{(1+\alpha r)^2} \frac{f(r)}{f'(r)} \right) =: f'(r) h(r,\alpha)
	\end{equation*}
	and $h(r,1)=1$, it suffices to show that  ${\partial h(r,\alpha)}/{\partial \alpha}<0$ at $\alpha=1$ in order to establish that $h(r,1-\epsilon)>1$, and hence $g'(r)>f'(r)>0$. But at $\alpha=1$ and for $r>r_\psi$, we have:
	\begin{align*}
	\frac{\partial h(r, \alpha)}{\partial \alpha } &= \frac{2}{(1+r)^2} \left( \frac{1-r^2}{2} -\frac{f(r)}{f'(r)} \right)\\
	                                               &= \frac{2}{(1+r)^2} \left( \frac{1-r^2}{2} -\frac{r}{\psi(-r)} \right)\\
	                                               &<0,
	\end{align*}
	using the equations~\eqref{r0}, \eqref{m3} and the fact that $k(r)<0$ for all $r\in (r_{\psi},1].$ That completes the proof.
\end{proof}

\begin{remark} The following result was proved by MacGregor~\cite{mc}:
	Let $g\in \mathcal{A}$ and $f\in \mathcal{C}$. If $g$ is majorized by $f$ in $\mathbb{D}$, then $|g'(z)| \leq |f'(z)|$ in $|z|\leq {1}/{3}.$
	The result is sharp.
\end{remark}

In our next result, we show the application to the Janowski class, which covers many well-known classes:
\begin{corollary}\label{AB}
	Let $f$ belongs to $\mathcal{C}[D,E]$, where $-1\leq E<D\leq1$ along with $1+D/E\geq0$ and $-1\leq E<0$. If $g$ is majorized by $f$, then
	\begin{equation*}
	|g'(z)| \leq |f'(z)| \quad\text{in}\quad |z|\leq r_0,
	\end{equation*}
	where $r_0$ is the smallest positive root of the equation
	$$(1-r^2)\left( {}_{2}F_{1}\left(1-\frac{D}{E}, 1, 2; \frac{-Er}{1-Er} \right) \right)^{-1}-2r=0.$$
	The result is sharp.
\end{corollary}
\begin{proof}
	In Theorem~\ref{mj}, put $\phi(z)= (1+D z)/(1+E z)$. Then we have $\psi(z):= 1/q(z)$, where
	\begin{equation*}
	q(z)=
	\left\{
	\begin{array}{lr}
	\int_{0}^{1}\left(\frac{1+Etz}{1+Ez}\right)^{\frac{D-E}{E}}dt, & \text{if}\; E\neq0;\\\\

	\int_{0}^{1}e^{D(t-1)z}dt, & \text{if}\; E=0,
	\end{array}
	\right.
	\end{equation*}
	which further can be represented in terms of confluent and Gaussian hypergeometric functions, respectively as follows:
	\begin{equation*}
	q(z)=
	\left\{
	\begin{array}{lr}
	{}_{2}F_{1}\left(1-\frac{D}{E}, 1, 2; \frac{Ez}{1+Ez} \right), & \text{if}\; E\neq0;\\
	
	{}_{1}F_{1}\left(1,2;-Dz\right), & \text{if}\; E=0.
	\end{array}
	\right.
	\end{equation*}
	Since $1+D/E\geq0$ and $-1\leq E<0$, therefore we have
	\begin{equation*}
	\min_{|z|=r}\Re\psi(z) = {\psi(-r)}=\frac{1}{q(-r)}= \left( {}_{2}F_{1}\left(1-\frac{D}{E}, 1, 2; \frac{-Er}{1-Er} \right) \right)^{-1}.
	\end{equation*}
	Since $\Re\psi(z)>0$ and $\min_{|z|=r}\Re\psi(z) = {\psi(-r)}$, therefore we conclude that
	$\min_{|z|=r}|\psi(z)|=\psi(-r)$
	and hence, the result follows from Theorem~\ref{mj}.
\end{proof}

Now we have the result for the class of convex functions of order $\alpha$ using Corollary~\ref{AB}:
\begin{corollary}
	Let $f$ belongs to $\mathcal{C}[1-2\alpha,-1]$, where $0\leq \alpha<1$. If $g$ is majorized by $f$, then $|g'(z)| \leq |f'(z)|$ in $|z|\leq r_0,$
	where $r_0$ is the smallest positive root of the equation
	$$(1-r^2)\left( {}_{2}F_{1}\left(2(1-\alpha), 1, 2; \frac{r}{1+r} \right) \right)^{-1}-2r=0.$$
	The result is sharp.
\end{corollary}

\begin{corollary}
	Let $f$ belongs to $\mathcal{C}[D,0]$. If $g$ is majorized by $f$, then $|g'(z)| \leq |f'(z)|$ in $|z|\leq r_0,$
	where $r_0$ is the smallest positive root of the equation
	$$(1-r^2)(Dr e^{-Dr}/(e^{-Dr}-1))+2r=0.$$
	The result is sharp.
\end{corollary}
\begin{proof}
	From the proof of Corollary~\ref{AB}, we obtain that $\psi(z)=Dz e^{Dz}/(e^{Dz}-1)$, when $\phi(z)=1+Dz$. Now with a little computation, we find that the function $l(z)=ze^z/(e^z-1)$ is convex univalent in $\mathbb{D}$. Therefore, the function $\psi(z)=l(Dz)$ is also convex in $\mathbb{D}$ for each fixed $0<D\leq1$. Since $\psi$ is also symmetric about the real axis, we conclude that $\min_{|z|=r}|\psi(z)|=\psi(-r)$. Hence the result.
\end{proof}

\begin{theorem}\label{hallen}
	Let $\phi$ be convex in $\mathbb{D}$, with $\Re \phi(z)>0$, $\phi(0)=1$ and suppose $f\in \mathcal{A}$ satisfies the differential subordination
	\begin{equation}\label{halb}
	\frac{zf'(z)}{f(z)}+z\left(\frac{zf'(z)}{f(z)}\right)' \prec \phi(z).
	\end{equation}
	If $g$ is majorized by $f$, then $|g'(z)|\leq |f'(z)|$ in $|z|\leq r_0$,
	where $r_0$ is the least positive root of the equation
	\begin{equation*}
	(1-r^2)\min_{|z|=r}\Re\psi(z)-2r=0,
	\end{equation*}
	where $$\psi(z):=\frac{1}{z}\int_{0}^{z}{\phi(t)}dt.$$
	The result is sharp for the case $\min_{|z|=r}\Re\psi(z)=\psi(\pm r)$.
\end{theorem}
\begin{proof}
	Let $p(z)=zf'(z)/f(z)$. Then the subordination \eqref{halb} can equivalently be written as:
	\begin{equation*}
	p(z)+zp'(z) \prec \phi(z).
	\end{equation*}
	A simple calculation show that the analytic function $\psi(z):=({1}/{z})\int_{0}^{z}{\phi(t)}dt$ satisfies
	\begin{equation*}
	\psi(z)+z\psi'(z)=\phi(z).
	\end{equation*}
	Now from the Hallenbeck and Ruscheweyh result \cite[Theorem~ 3.1b, p.~71]{subbook}, we have $p\prec \psi$, where $\psi$ is the best dominant and also convex. Further, since $\Re\phi(z)>0$, using the integral operator \cite[Theorem~ 4.2a, p.~202]{subbook} preserving functions with positive real part, we see that $\psi$ is a Carathe\'{o}dory function. Thus we have
	$$\frac{f(z)}{zf'(z)} \prec \frac{1}{\psi(z)} \; \quad \text{which implies }\quad \left|\frac{f(z)}{f'(z)}\right|\leq \frac{r}{\min_{|z|=r}|\psi(z)|}= \frac{r}{\min_{|z|=r}\Re\psi(z)}.$$
	Now proceeding same as in the Theorem~\ref{mj} result follows.
\end{proof}		

\begin{corollary}
	Suppose $f\in \mathcal{A}$ satisfies the differential subordination
	\begin{equation*}
	\frac{zf'(z)}{f(z)}+z\left(\frac{zf'(z)}{f(z)}\right)' \prec \frac{1+z}{1-z}.
	\end{equation*}
	If $g$ is majorized by $f$, then $|g'(z)|\leq |f'(z)|$ in $|z|\leq r_0$, where $r_0$ is the least positive root of the equation
	$$(1-r^2)(2\log(1+r)-r)-2r^2=0.$$
	The result is sharp.
\end{corollary}

\begin{corollary}
	Suppose $f\in \mathcal{A}$ satisfies the differential subordination
	\begin{equation*}
	\frac{zf'(z)}{f(z)}+z\left(\frac{zf'(z)}{f(z)}\right)' \prec e^z.
	\end{equation*}
	If $g$ is majorized by $f$, then $|g'(z)|\leq |f'(z)|$ in $|z|\leq r_0$, where $r_0$ is the least positive root of the equation
	$$(1-r^2)(1-e^{-r})-2r^2=0.$$
	The result is sharp.
\end{corollary}

\begin{theorem}
	Let $\phi$ be convex in $\mathbb{D}$, with $\Re \phi(z)>0$, $\phi(0)=1$ and suppose $f\in \mathcal{A}$ satisfies the differential subordination
	\begin{equation}\label{p2p}
	\frac{zf'(z)}{f(z)}\left(\frac{zf'(z)}{f(z)}+2z\left(\frac{zf'(z)}{f(z)}\right)'\right) \prec \phi(z),\quad \alpha\in [0,1).
	\end{equation}
	If $g$ is majorized by $f$, then $|g'(z)|\leq |f'(z)|$ in $|z|\leq r_0$,
	where $r_0$ is the least positive root of the equation
	\begin{equation*}
	(1-r^2)\min_{|z|=r}|\sqrt{\psi(z)}|-2r=0,
	\end{equation*}
	where $$\psi(z):=\frac{1}{z}\int_{0}^{z}{\phi(t)}dt.$$
	The result is sharp when $\min_{|z|=r}|\sqrt{\psi(z)}|= \sqrt{\psi(\pm r)}$.
\end{theorem}
\begin{proof}
	Let $p(z)=zf'(z)/f(z)$. Then the subordination \eqref{p2p} can be equivalently written as:
	\begin{equation*}
	p^2(z)+2 zp(z)p'(z) \prec \phi(z),
	\end{equation*}
	which using the change of variable $P(z)=p^2(z)$ becomes
	\begin{equation*}
	P(z)+zP'(z) \prec \phi(z).
	\end{equation*}
	Now proceeding as in Theorem~\ref{hallen}, we see that $p(z) \prec \sqrt{\psi(z)}$ and $\sqrt{\psi(z)}$ is the best dominant. Further, since $\Re\phi(z)>0$, using \cite[Theorem~ 4.2a, p.~202]{subbook}, we see that $\psi$ is a Carathe\'{o}dory function. Therefore,
	$$|\arg{\sqrt{\psi(z)}}|=\frac{1}{2}|\arg{\psi(z)}| \leq \frac{\pi}{4},$$
	which implies $\Re\sqrt{\psi(z)}>0$. Thus we have
	\begin{equation*}
	\frac{f(z)}{zf'(z)} \prec \frac{1}{\sqrt{\psi(z)}} \quad \text{which implies} \quad \left|\frac{f(z)}{f'(z)}\right| \leq \frac{r}{\min_{|z|=r}|\sqrt{\psi(z)}|}.
	\end{equation*}
	Now proceeding same as in the Theorem~\ref{mj} result follows.
\end{proof}		

\begin{corollary}
	Suppose $f\in \mathcal{A}$ satisfies the differential subordination
	\begin{equation*}
	\frac{zf'(z)}{f(z)}\left(\frac{zf'(z)}{f(z)}+2z\left(\frac{zf'(z)}{f(z)}\right)'\right) \prec \frac{1+(2\alpha -1)z}{1+z}.
	\end{equation*}
	If $g$ is majorized by $f$, then $|g'(z)|\leq |f'(z)|$ in $|z|\leq r_0$,
	where $r_0$ is the least positive root of the equation
	\begin{equation*}
	(1-r^2)\min_{|z|=r}\Re\sqrt{\psi(z)}-2r=0,
	\end{equation*}
	where $$\psi(z):=\frac{1}{z}\int_{0}^{z}{\frac{1+(2\alpha -1)t}{1+t}}dt.$$	
\end{corollary}

\begin{corollary}
	Suppose $f\in \mathcal{A}$ satisfies the differential subordination
	\begin{equation*}
	\frac{zf'(z)}{f(z)}\left(\frac{zf'(z)}{f(z)}+2z\left(\frac{zf'(z)}{f(z)}\right)'\right) \prec 1+\alpha z, \quad (\alpha\in (0,1]).
	\end{equation*}
	If $g$ is majorized by $f$, then $|g'(z)|\leq |f'(z)|$ in $|z|\leq r_0$,
	where $r_0$ is the least positive root of the equation
	\begin{equation*}
	(1-r^2)\sqrt{1-\beta r}-2r=0, \quad\text{where}\quad \beta=\alpha/2.
	\end{equation*}
	The result is sharp.
\end{corollary}

Now we state the following result without proof as it follows from Theorem~\ref{mj}:
\begin{theorem}\label{mj-s}
	Let $\psi\in\mathcal{P}$ be a univalent function
	such that  $$m_r:=\underset{|z|=r}{\min}|\psi(z)|=\left\{
	\begin{array}{ll}
	\psi(-r), & \hbox{ if } \psi'(0)>0;  \\
	\psi(r), & \hbox{ if } \psi'(0)<0.
	\end{array}
	\right.$$
	Let $g\in \mathcal{A}$ and $f\in \mathcal{S}^*(\psi)$. If $g$ is majorized by $f$ in $\mathbb{D}$, then
	\begin{equation*}\label{m-s}
	|g'(z)| \leq |f'(z)| \quad \text{in}\;|z|\leq r_{\psi},
	\end{equation*}
	where $r_{\psi}$ is the least positive root of the equation
	\begin{equation*}\label{r0-s}
	(1-r^2)m_r-2r=0.
	\end{equation*}
	The result is sharp.
\end{theorem}	

\begin{example}
	Let us consider the analytic functions $\psi_1(z)=\sqrt{1-z}$ and $\psi_2(z)=\sqrt{1+z}$. Note that ${\psi}'_1(0)<0$, ${\psi}'_2(0)>0$  and for $|z|=r$,
	$$m_{r_1}=\min_{|z|=r}|\psi_1(z)|=\psi_1(r)=\sqrt{1-r}=\psi_2(-r)=\min_{|z|=r}|\psi_2(z)|=m_{r_2}.$$
	Now from Theorem \ref{mj-s}, we obtain the following result:\\
	If $g\in\mathcal{A}$, $f\in \mathcal{S}^{*}(\psi_i)$, where $i=1,2$ and $g$ is majorized by $f$, then $|g'(z)|\leq |f'(z)|$ in $|z|\leq r_0$, where $r_0$ is the least positive root of the equation
	$$(1-r^2)\sqrt{1-r}-2r=0.$$
	Interestingly, the desired radius in both the cases remain same as $\psi_1(\mathbb{D})=\psi_2(\mathbb{D})$, though  $\psi_1$ and $\psi_2$ are oppositely oriented.	
\end{example}	

\begin{remark}\label{alpha-eta}
	Taking $\alpha=0$ or $\eta=1$ in corollary~\ref{all-result}, case $(ii)$ and $(iii)$, respectively, we obtain the result proved by T. H. MacGregor \cite{mc}, namely:
	{\it Let $g\in \mathcal{A}$ and $f\in \mathcal{S}^*$. If $g<<f$ in $\mathbb{D}$, then
		$|g'(z)| \leq |f'(z)|\; \text{in}\; |z|\leq 2-\sqrt{3}.$} The result is sharp.
\end{remark}

Now we obtain the following majorization results for some known classes as well those introduced and studied in \cite{kumar-2019,Goel,mendi2exp,raina-2015}.
\begin{corollary}\label{all-result}
	Let $g\in \mathcal{A}$ and $f\in \mathcal{S}^*(\psi)$. If $g<<f$ in $\mathbb{D}$, then $|g'(z)|\leq |f'(z)|$ in $|z|\leq r_{\psi}$, where  $r_{\psi}$ is the least positive root of the equation
	$P(r)=0$ 	and the result follows for each one of the following cases:
	\begin{itemize}
		\item [$(i)$]  	$P(r)=(1-r^2)((1-Dr)/(1-Er))-2r$ when $\psi(z)= \frac{1+Dz}{1+Ez}$, where $-1\leq E<D\leq1$.
		
		\item [$(ii)$]	$P(r)=(1-r)(1-(1-2\alpha)r)-2r$ when $\psi(z)=\frac{1+(1-2\alpha)z}{1+z}$, where $0\leq\alpha<1$.
		
		\item [$(iii)$] $P(r)=(1-r^2)((1-r)/(1+r))^{\eta}-2r$ when $\psi(z)=\left(\frac{1+z}{1-z}\right)^{\eta}$, where $0<\eta \leq1$.
		
		\item [$(iv)$] 	$P(r)=(1-r^2)\left(\sqrt{2}-(\sqrt{2}-1)\sqrt{\frac{1+r}{1-2(\sqrt{2}-1)r}}\right)-2r$ when  $\psi(z)=\sqrt{2}-(\sqrt{2}-1)\sqrt{\frac{1-z}{1+2(\sqrt{2}-1)z}}$.
		
		\item [$(v)$] 	$P(r)=(1-r^2)(b(1-r))^{1/a}-2r$ when $\psi(z)=(b(1+z))^{1/a}$, where $a\geq1$ and $b\geq 1/2$.

		\item [$(vi)$] $P(r)=(1-r^2)-2re^{r}$ when $\psi(z)=e^z$.

		\item [$(vii)$] $P(r)=(1-r^2)(\sqrt{1+r^2}-r)-2r$ when $\psi(z)=z+\sqrt{1+z^2}$.

		\item [$(viii)$] $P(r)=(1-r^2)-r(1+e^r)$ when $\psi(z)=\frac{2}{1+e^{-z}}$.
		
		\item [$(ix)$] \label{sin} $P(r)=(1-r^2)(1-\sin{r})-2r$ when $\psi(z)=1+\sin{z}$.
		
	\end{itemize}
The results are sharp.
\end{corollary}

\begin{remark}
	In Corollary~\ref{all-result}, case $(ix)$, we obtained the radius $r_{\psi}\approx0.312478$ which improves the majorization-radius $r_s\approx0.309757$ obtained in \cite{Tang-HM-2019}.
\end{remark}

Let $\psi(z)=1+z/(1-\alpha z^2)$, $0\leq \alpha<1$,  introduced and studied by Kargar et al. \cite{kargar-2019}.  Clearly $\psi\in \mathcal{P}$ only when $\alpha=0$ and hence Theorem \ref{mj-s} holds when $\psi(z)=1+z$. Moreover, for some $r>0$, the quantity  $z/\psi(z)$ does not exist for all $|z|=r$. In view of the same, the result proved by  Tang and  Deng \cite{tang}, needs correction and the corrected version is stated in the following corollaries:
\begin{corollary}
	Let $g\in \mathcal{A}$ and $f\in \mathcal{S}^*(1+\beta z)$, $0<\beta\leq1$. If $g<<f$ in $\mathbb{D}$, then
	$$|g'(z)| \leq |f'(z)|\quad \text{in}\quad |z|\leq r_{\beta},$$
	where $r_{\beta}$ is the least positive root of the equation
	$$(1-r^2)(1-\beta r)-2r=0.$$
	The result is sharp.
\end{corollary}	

Now we obtain the result related to $\mathcal{BS}(\alpha)$, the class of Booth lemniscate starlike functions, when $\alpha\neq0$.
\begin{corollary}
	Let $0<\alpha<1$  and $r_{\alpha}$ be the unique root of the equation
	\begin{equation}\label{boothroot}
	\alpha r^2+r-1=0.
	\end{equation}
	Let $g\in \mathcal{A}$ and $g<< f$ in $\mathbb{D}$, where $f\in\mathcal{BS}(\alpha)$. Then $$|g'(z)|\leq|f'(z)| \quad\text{in}\quad |z|\leq r_{B(\alpha)}:=\min\{r_{\alpha}, r_0\},$$ where
	$r_0$ is the least positive root of the equation
	$$(1-r^2)\left(1-\frac{r}{1-\alpha r^2}\right)-2r=0.$$
	The result is sharp.
\end{corollary}
\begin{proof}
	Observe that $\Re\left(1+\frac{z}{1-\alpha z^2}\right)>0$ for $|z|<r_\alpha$, where $r_\alpha$ is the unique root of \eqref{boothroot}. Thus the inequality in \eqref{m1} holds for $|z|=r<r_{\alpha}$ and the result follows at once.
\end{proof}	
	
	\section{\bf{Product of starlike functions and a sufficient condition }}\label{sec-2}
	Assume that $\psi_1$ and $\psi_2$ belong to $\mathcal{P}$ and satisfy the following  conditions for $|z|=r$ and $i=1,2$
	\begin{equation}\label{p}
	\max_{|z|=r}\Re\psi_i(z)= \psi_i(r) \quad \text{and} \quad \min_{|z|=r}\Re\psi_i(z)= \psi_i(-r).
	\end{equation}
	
	Motivated by Obradovi\'{c} and Ponnusamy \cite{obPonnu}, in this section, we consider the radius problem to generalize their result and also establish a similar result for the Urlagaddi class $\mathcal{M}(\beta):=\{f\in\mathcal{A}: \Re(zf'(z)/f(z))<\beta,\; \beta>1\}$. Also we extend a result of the Bulboac\u{a} and Tuneski~\cite{Bulboca-2003} for the class $\mathcal{S}^*(\psi)$.
	\begin{theorem}\label{thm-urlla}
		Let $g \in \mathcal{S}^*(\psi_1)$ and $h \in \mathcal{S}^*(\psi_2)$, where $\psi_{i}$ satisfy the first condition in \eqref{p}. Then the function $F$ defined by
		\begin{equation}\label{product}
		F(z)=\frac{g(z) h(z)}{z}
		\end{equation}
		belongs to $\mathcal{M}(\beta)$ in the disk $|z|< r_\beta=\min\{1,r_0(\beta)\}$, where $r_0(\beta)$ is the least positive root of the equation
		\begin{equation}\label{urlla-1}
		\psi_1(r)+\psi_2(r)-1-\beta=0.
		\end{equation}
		The radius $r_\beta$ is sharp.
	\end{theorem}
	\begin{proof}
		Let $g\in \mathcal{S}^*(\psi_1)$ and $h\in \mathcal{S}^*(\psi_2)$. Then in view of \eqref{p} and subordination principle, it follows that
		$$\Re\frac{zg'(z)}{g(z)}\leq\psi_1(r)\quad\text{and}\quad \Re\frac{zh'(z)}{h(z)}\leq\psi_2(r)$$
		in $|z|\leq r<1$. Since
		$$\frac{zF'(z)}{F(z)}=\frac{zg'(z)}{g(z)}+\frac{zh'(z)}{h(z)}-1,$$
		we have for $|z|=r$,
		$$\Re\frac{zF'(z)}{F(z)}\leq \psi_1( r)+\psi_2( r)-1\leq\beta,$$
		whenever $r\leq\min\{1,r_0(\beta)\}$, where $r_0(\beta)$ is the least positive root of the equation \eqref{urlla-1}.
		The sharpness follows by considering the functions
		$$	g(z)=z\exp{\int_{0}^{z}\frac{\psi_1(t)-1}{t}dt} \quad \text{and} \quad
		h(z)=z\exp{\int_{0}^{z}\frac{\psi_2(t)-1}{t}dt}.$$
	\end{proof}
	
	\begin{corollary}
		Let $g \in \mathcal{S}^*(\gamma)$ and $h\in \mathcal{S}^*(\tau)$. Then the function $F$ defined in \eqref{product} belongs to $\mathcal{M}(\beta)$ in the disk $|z|< \min\{1, r_0(\beta)\}$, where $$r_0(\beta)=\frac{\beta-1}{3+\beta-2(\gamma+\tau)}.$$
	\end{corollary}
	
	The proof of the following result is much akin to Theorem~\ref{thm-urlla}, so omitted here.
	\begin{theorem}\label{thm-oder}
		Let $g \in \mathcal{S}^*(\psi_1)$ and $h \in \mathcal{S}^*(\psi_2)$, where $\psi_{i}$ satisfy the second condition in \eqref{p}. Then the function $F$ defined in \eqref{product} is starlike of order $\gamma$ in the disk $|z|<r_{\gamma}$, where $r_{\gamma}$ is the least positive root of the equation
		\begin{equation*}\label{str-oder}
		\psi_1(- r)+\psi_2(- r)-1-\gamma=0.
		\end{equation*}
		The radius $r_{\gamma}$ is sharp.
	\end{theorem}
	%
	We obtain the following result proved by Obradovi\'{c} and Ponnusamy \cite{obPonnu}:
	\begin{remark}
		Let $g \in \mathcal{S}^*(\gamma)$ and $h\in \mathcal{S}^*(\tau)$. Then the function $F$ defined in \eqref{product} is starlike of order $\gamma_0$ in the disk
		$$|z|<\frac{1-\gamma_0}{\gamma_0 +3-2(\gamma+\tau)}.$$
	\end{remark}
	
	\begin{remark}
		Note that the identity function $z\in \mathcal{S}^{*}(\psi)$. Thus if we choose $g(z)=z$ (or $h(z)=z$) in \eqref{product}, then the problem reduces to obtaining the $\mathcal{M}(\beta)$-radius (or $\mathcal{S}^{*}(\gamma)$-radius) of the class $\mathcal{S}^*(\psi_2)$ (or $\mathcal{S}^*(\psi_1)$). It is also evident that the conditions given in \eqref{p} establish the inclusion relations $\mathcal{S}^{*}(\psi)\subseteq\mathcal{M}(\psi(1))$ and $\mathcal{S}^{*}(\psi)\subseteq\mathcal{S}^{*}(\psi(-1))$, respectively.
	\end{remark}
	
	In the following, we extend the results of the Bulboac\u{a} and Tuneski~\cite{Bulboca-2003}:
	\begin{theorem}
		Let $h$ be analytic with $h(0)=0$, $h'(0)\neq0$. Suppose that $h$ satisfies
		\begin{equation*}
		\Re\left(1+\frac{zh''(z)}{h(z)}\right) \geq -\frac{1}{2}
		\end{equation*}
		and
		\begin{equation}\label{p}
		\frac{1}{z}\int_{0}^{z}h(t)dt \prec \frac{\psi(z)-1}{\psi(z)}.
		\end{equation}
		If  $f\in \mathcal{A}$, then
		\begin{equation*}
		\frac{f(z)f''(z)}{(f'(z))^2} \prec h(z) \quad\text{implies} \quad f\in \mathcal{S}^{*}(\psi).
		\end{equation*}
	\end{theorem}
	\begin{proof}
		Using the result \cite[Theorem~3.1, p.~3]{Bulboca-2003}, we see that ${f(z)f''(z)}/{(f'(z))^2} \prec h(z)$ implies
		\begin{equation*}
		\frac{1}{z}\int_{0}^{z}\left(1-\left(\frac{f(t)}{f'(t)}\right)'\right)dt = 1-\frac{f(z)}{zf'(z)} \prec \frac{1}{z}\int_{0}^{z}h(t)dt.
		\end{equation*}	
		From the above subordination, we have
		\begin{equation*}
		\frac{f(z)}{zf'(z)} \prec	1-\frac{1}{z}\int_{0}^{z}h(t)dt.
		\end{equation*}
		Now to prove that $f\in \mathcal{S}^{*}(\psi)$, it suffices to consider
		\begin{equation*}
		1-\frac{1}{z}\int_{0}^{z}h(t)dt \prec \frac{1}{\psi(z)},
		\end{equation*}
		which is equivalent to \eqref{p}. This completes the proof. \qed
	\end{proof}	
	
	\section{\bf{Distortion theorem }}\label{sec-3}
	Ma-Minda \cite{minda94} proved the distortion theorem for the class $\mathcal{S}^*(\psi)$ with some restriction on $\psi$, namely $|\psi(z)|$ attains its maximum and minimum value  respectively at $z=r$ and $z=-r$, see eq.~\eqref{minda-cond}. Now what  if $\psi$ does not satisfy the condition \eqref{minda-cond} and why the condition \eqref{minda-cond} is so important? To answer this, we first need to recall the following result:
	\begin{lemma}\emph{(\cite{minda94})}\label{grth}
		Let $f\in \mathcal{S}^*(\psi)$ and $|z_0|=r<1$. Then $-f_0(-r)\leq|f(z_0)|\leq f_0(r).$
		Equality holds for some $z_0\neq0$ if and only if $f$ is a rotation of $f_0$, where $zf_0(z)/f_0(z)=\psi(z)$ such that
		\begin{equation}\label{int-rep}
		f_0(z)=z\exp{\int_{0}^{z}\frac{\psi(t)-1}{t}dt}.
		\end{equation}
		
	\end{lemma}
	
	We see that a  Ma-Minda starlike function, in general, need not satisfy the condition~\eqref{minda-cond}.  To examine the same, let us consider two different  Ma-Minda starlike functions, namely
	$\psi_1(z):=z+\sqrt{1+z^2}$ and $\psi_2(z):=1+ze^z.$
	The unit disk images of $\psi_1$ and $\psi_2$ are displayed in figure~\ref{fg1} and figure~\ref{fg2}.

		\begin{figure}[h]
		\begin{tabular}{c}
\includegraphics[scale=0.30]{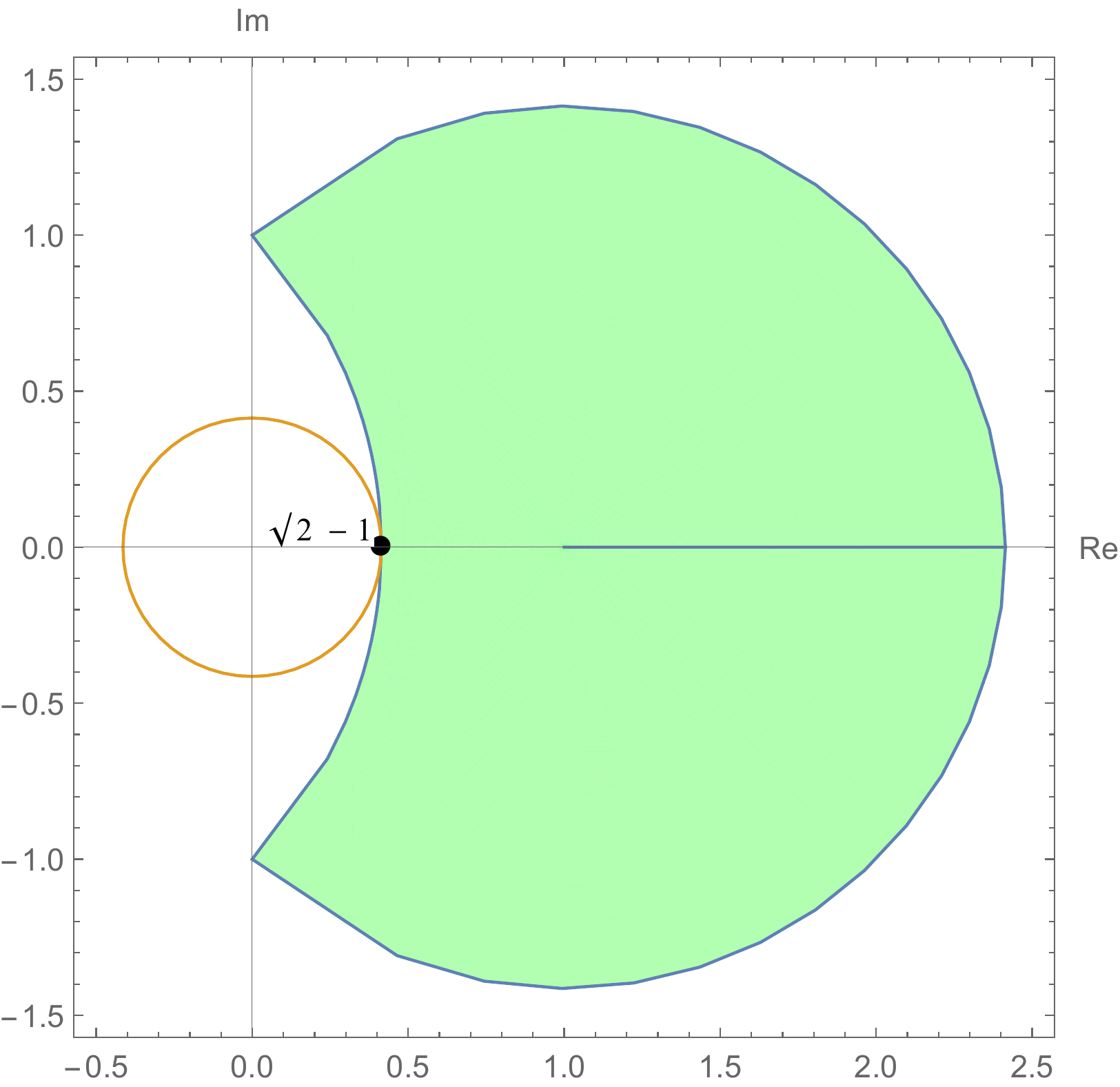}
		\end{tabular}
		\begin{tabular}{l}
	\parbox{0.5\linewidth}{
	{\bf \underline{Legend for figure~\ref{fg3}} -}\\~\\
	Let $\underset{|z|=r}{\min}|\psi_2(z)|=:\gamma_i(r),$ where\\
	$z=r_i e^{i\theta},\; 0 \leq \theta\leq\pi$,\\
	then from table~\ref{table}, we have\\
	$\gamma_1(1)= 0.372412,  $\\
	$\gamma_2(4/5)= 0.527912, $ \\
	$\gamma_3(2/3)= 0.611553,$ \\
	$\gamma_4(1/2)= 0.693287,$ \\
	$\gamma_5(r)= 1-re^{-r},$
	where $r\leq(3-\sqrt{5})/2.$
	}
		\end{tabular}
{\caption{$\psi_1(z):=z+\sqrt{1+z^2}$}\label{fg1}}
	\end{figure}
	
\begin{figure}[h]
	{\includegraphics[scale=0.195]{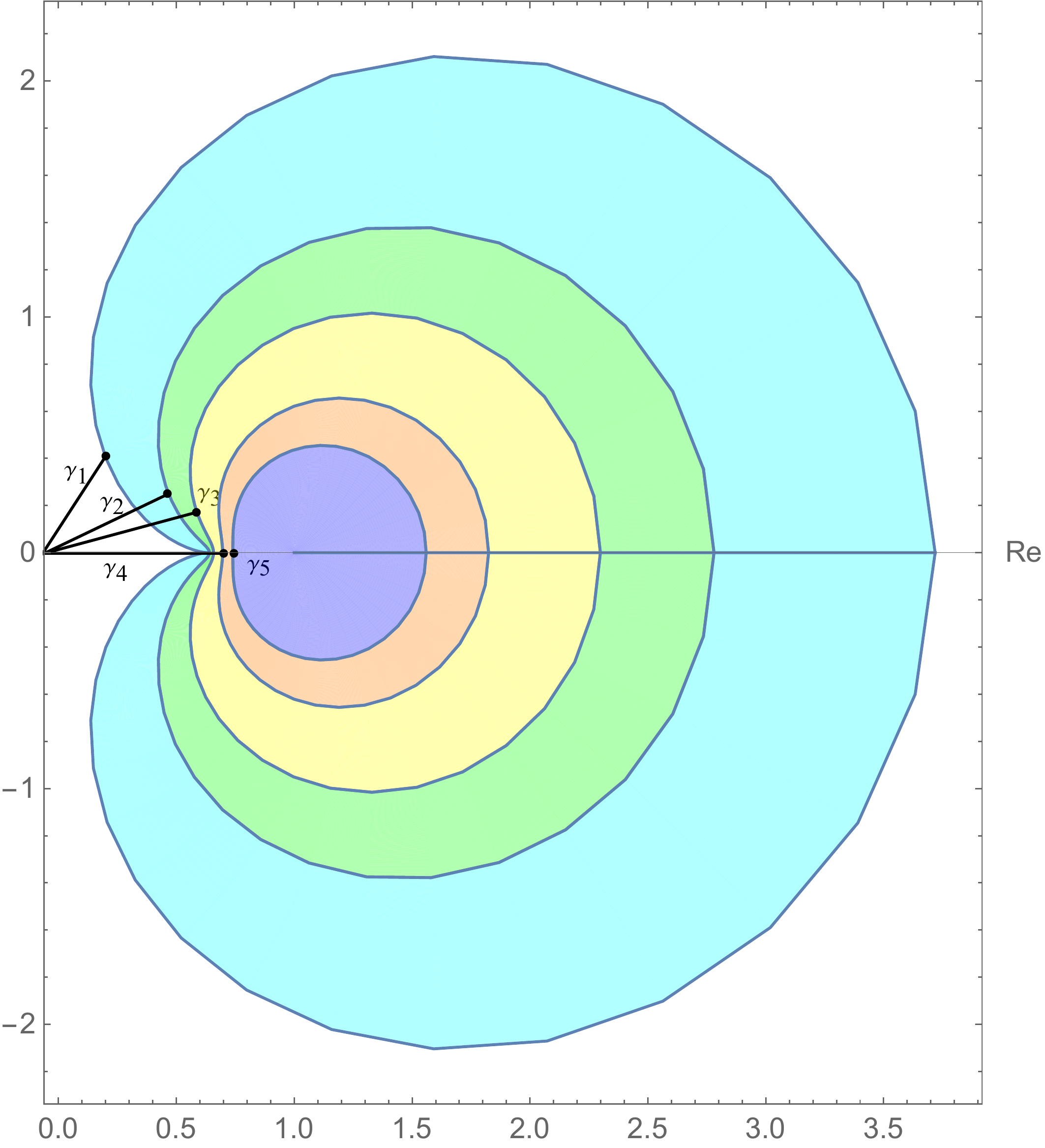}}{\caption{$\psi_2(z):=1+ze^z$}\label{fg2}}
\end{figure}
\begin{figure}[h]
	{\includegraphics[scale=0.55]{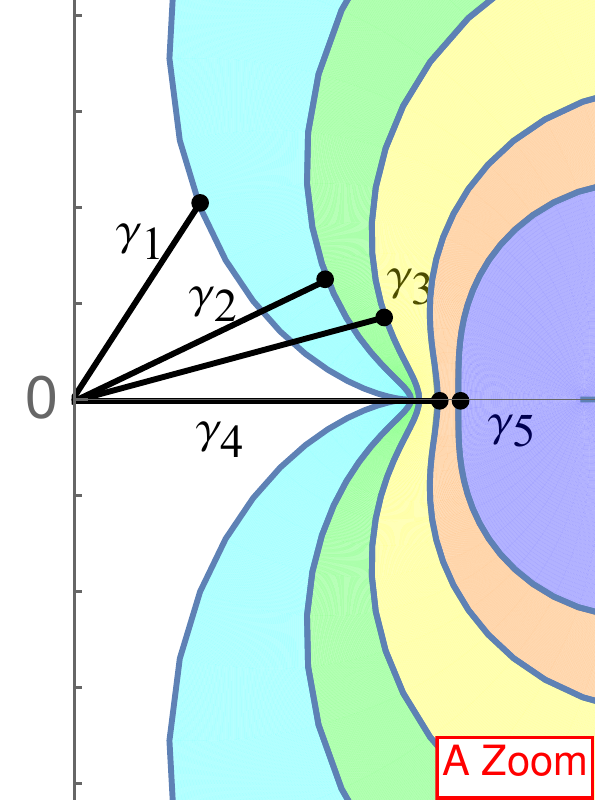}}{\caption{A zoom of figure~\ref{fg2}\label{fg3}}}	
\end{figure}	
	We know that the radius of a circle, centered at origin and touching only the boundary points of a image domain of a complex function, yields the optimal values of the modulus of the function. For example, see figure~\ref{fg1} to locate the lower bound of the modulus for a crescent function. Therefore it  is evident from figure~\ref{fg1}   that both the bounds $\psi_1(-r)$ and $\psi_1(r)$ of $|\psi_1|$ are attained on the real line and we have $\psi_1(-r)\leq|\psi_1(z)|\leq \psi_1(r)$ for each $|z|=r$. Whereas, from figure~\ref{fg2}, we see that only the upper bound $\psi_2(r)$ of $|\psi_2|$  is attained  on the real line and $|\psi_2(z)|\leq \psi_2(r)$ for each $|z|=r$. Although both $\psi_1$ and $\psi_2$ are Ma-Minda functions but the distortion theorem of Ma-Minda \cite[theorem~2, p.~162]{minda94} does not accommodate the function $\psi_2$, as the lower bound for $|\psi_2(z)|$ is not attained on the real line for all $|z|=r> (3-\sqrt{5})/2$, see figure~\ref{fg3}. To overcome this limitation, we modify the distortion theorem, wherein we theoretically assume the modulus bounds of the function and obtain a more general result. Thus the Ma-Minda functions, for which modulus bounds are not attained on the real line but could be estimated,  can now be entertained  for distortion theorem using the following result:
	\begin{theorem}[Modified Distortion Theorem]\label{distthm}
		Let $\psi$ be a Ma-Minda function. Assume that $\underset{|z|=r}{\min}|\psi(z)|=|\psi(z_1)|$ and $\underset{|z|=r}{\max}|\psi(z)|=|\psi(z_2)|$, where $z_1=re^{i\theta_1}$ and $z_2=re^{i\theta_2}$ for some $\theta_1, \theta_2\in [0,\pi]$.
		Let $f\in \mathcal{S}^*(\psi)$ and $|z_0|=r<1$. Then
		\begin{equation}\label{disteq}
		|\psi(z_1)|\left(\frac{-f_0(-r)}{r}\right)\leq|f'(z_0)|\leq \left(\frac{f_0(r)}{r}\right)|\psi(z_2)|.
		\end{equation}
	\end{theorem}
	\begin{proof}
		Let $p(z)=zf'(z)/f(z)$. Then $f\in \mathcal{S}^*(\psi)$	 if and only if $p(z)\prec \psi(z)$. Using a result \cite[Theorem~1, p.161]{minda94}, we have
		\begin{equation}\label{dist-equation}
		\frac{f(z)}{z}\prec \frac{f_0(z)}{z},
		\end{equation}
		where $f_0$ is given by \eqref{int-rep}.
		Now using Maximum-Minimum principle of modulus, \eqref{dist-equation} and by Lemma \ref{grth}, $-f_0(-r)/r\leq|f(z_0)/z|\leq f_0(r)/r$, we easily obtain for $|z_0|=r$
		\begin{align*}
		|\psi(z_1)|\left(\frac{-f_0(-r)}{r}\right)&=\min_{|z|=r}|\psi(z)|\min_{|z|=r}\left|\frac{f_0(z)}{z}\right|\\
		&\leq\left|p(z_0)\frac{f(z_0)}{z_0}\right|\\
		&=|f'(z_0)|\leq\max_{|z|=r}|\psi(z)|\max_{|z|=r}\left|\frac{f_0(z)}{z}\right|\\
		&= \left(\frac{f_0(r)}{r}\right)|\psi(z_2)|,
		\end{align*}
		that is,
		\begin{equation*}
		|\psi(z_1)|\left(\frac{-f_0(-r)}{r}\right)\leq|f'(z_0)|\leq\left(\frac{f_0(r)}{r}\right)|\psi(z_2)|,
		\end{equation*}
		where $z_1$ and $z_2$ are as defined in the hypothesis. Hence the result.\qed
	\end{proof}

	To illustrate Theorem~~\ref{distthm}, we consider the function $\psi(z)=1+ze^z$. Then we have the following expression for its modulus:
	\begin{equation}\label{cardioid-dist}
	|\psi(z)|=\sqrt{1+ r e^{r\cos\theta}(r e^{r\cos\theta} + 2\cos(\theta+r\sin\theta))}.
	\end{equation}
	Using equation~\eqref{cardioid-dist} and Theorem~\ref{distthm}, we obtain the following table~\ref{table}, providing the minimum for various choices of $r$.
	\begin{table}[h]
		\caption{The lower bounds for $ |1+ze^z| $ for different choices of $r=|z|$. }\label{table}
		\centering	
		\begin{tabular}{|c|c|c|c|}
			\hline
			 $r$ & $0\leq \theta_1 \leq \pi$ & $ |\psi(re^{i\theta_1 })|$ & $m(r,\theta_1 )=|\psi(re^{i\theta_1 })|(-f_0(-r)/r)$    \\
			\hline
			1 &  1.88438 & 0.372412  &0.197923 \\
			\hline
			 4/5 &  2.01859 & 0.527912 & 0.304374\\
			\hline
			2/3 & 2.17677 & 0.611553 & 0.375966\\
			\hline
				1/2 & 2.58169 & 0.693287 & 0.467769 \\
			\hline
			$r\leq (3-\sqrt{5})/2$ & $\pi$ & $\psi_2(-r)$ &$f'_0(-r)$ \\
			\hline
		\end{tabular}
	\end{table}	
	
	Now using Theorem~\ref{distthm}, we obtain the following distortion theorem for the class $\mathcal{S}^{*}(1+ze^z)$:
	\begin{example}
		Let $\psi(z)=1+ze^z$ and $\underset{|z|=r}{\min}|\psi(z)|=|\psi(z_1)|$, where $z_1=re^{i\theta_1}$ for some $\theta_1\in [0,\pi]$.
		Let $f\in \mathcal{S}^*(\psi)$ and $|z_0|=r<1$. Then
		\begin{equation*}
		m(r,\theta_1)\leq|f'(z_0)|\leq f'_0(r), \quad \left(r>\tfrac{3-\sqrt{5}}{2}\right)
		\end{equation*}
		and
		\begin{equation*}
		f'_0(-r)\leq|f'(z_0)|\leq f'_0(r), \quad \left(r\leq\tfrac{3-\sqrt{5}}{2}\right),
		\end{equation*}
		where $f_0(z)= z \exp(e^z-1)$ and $m(r,\theta_1)$ is provided in table~\ref{table} for some specific values of $r$.
	\end{example}
	
	\begin{remark}
		In Theorem \ref{distthm}, if we assume that $\theta_1=\pi$ and $\theta_2=0$, then extremes in equation \eqref{disteq} simplifies to $f'_0(-r)$ and $f'_0(r)$, respectively since $zf'_0(z)/f_0(z)=\psi(z)$. Thus the extremes in the equation \eqref{disteq} are in terms of $r$ alone and also lead to the sharp bounds.
		Consequently, we obtain the following distortion theorem of Ma-Minda \cite{minda94} as a special case of  Theorem~\ref{distthm}:\\
		
		{\it	Let $\min_{|z|=r}|\psi(z)|=\psi(-r)$ and $\max_{|z|=r}|\psi(z)|=\psi(r)$. If $f\in \mathcal{S}^*(\psi)$ and $|z_0|=r<1$. Then
			\begin{equation*}
			f'_0(-r)=\psi(-r)\frac{f_0(-r)}{-r}\leq|f'(z_0)|\leq \frac{f_0(r)}{r}\psi(r)=f'_0(r).
			\end{equation*}
			Equality holds for some $z_0\neq0$ if and only if $f$ is a rotation of $f_0$.}	
	\end{remark}

	\section{Bohr-Phenomenon for functions in $S_{f}(\psi)$}\label{sec-4}
	
	Note that {``the Bohr radius of the class $\mathcal{X}$ is at least $r_{x}$'',} this holds for every result in this section. In general, Bohr radius is estimated for a specific class provided the sharp coefficients bounds of the functions in that class are known. For instance, consider the class of starlike univalent functions, where we have the sharp coefficient bounds: $|a_n|\leq n$. However, for most of the Ma-Minda subclasses, the better coefficients bounds are yet not known. Hence, we encounter the following problem, especially in context of Ma-minda classes, which we deal here to a certain extent:\\
	{\bf Problem: {\it If coefficients bounds are not known, how one can find a good lower estimate for the Bohr radius of a given class?}}
	
	To readily understand the above problem, consider the class $\mathcal{S}^{*}(1+ze^z)$, where the sharp coefficients bounds for functions in this class are unknown. In this situation, how one can find the Bohr radius for this class or is there any way out with the lower bounds all alone? Here below we state Theorem~\ref{bohr}, where we find a solution for this problem. Note that the Bohr radius $3-2\sqrt{2}\approx 0.1713$ for the class $\mathcal{S}^{*}$ serves as a lower bound for the class $S_{f}(\psi)$ and is also a special case of Theorem~\ref{bohr}.
	
	Let $\mathbb{B}(0,r):=\{z\in\mathbb{C}: |z|<r  \},$  $g(z)=\sum_{k=1}^{\infty}b_k z^k,$   $\mathcal{S}^*(\psi)$
	and   $S_{f}(\psi)$ as defined in  \eqref{mindaclass} and \eqref{bohrclass} respectively.
	For any $g\in S_{f}(\psi),$ we find the radius $r_b$ so that $S_{f}(\psi)$ obey the following Bohr-phenomenon:
	\begin{equation}\label{bohrphenomenon}
	\sum_{k=1}^{\infty}|b_k|r^k \leq d(f(0),\partial\Omega)\quad \text{for}\quad|z|=r\leq r_b,
	\end{equation}
	where  $d(f(0),\partial\Omega)$ denotes the Euclidean distance between $f(0)$ and the boundary of $\Omega=f(\mathbb{D})$. Now we prove our main result:

	\begin{theorem}\label{bohr}
		Let $r_{*}$ be the Koebe-radius for the class $\mathcal{S}^*(\psi),$  $f_0(z)$ be given by the equation~\eqref{int-rep} and $g(z)=\sum_{k=1}^{\infty}b_k z^k \in S_{f}(\psi)$. Assume  $f_0(z)=z+\sum_{n=2}^{\infty}t_n z^n$ and $\hat{f}_0(r)=r+\sum_{n=2}^{\infty}|t_n|r^n$.
		Then   $S_{f}(\psi)$ satisfies the Bohr-phenomenon
		\begin{equation}\label{compare}
		\sum_{k=1}^{\infty}|b_k|r^k \leq d(f(0),\partial\Omega),\quad \text{for}\; |z|=r\leq r_b,
		\end{equation}
		where $r_b=\min\{r_0, 1/3 \}$, $\Omega=f(\mathbb{D})$ and $r_0$ is the least positive root of the equation
		$$\hat{f}_0(r)=r_{*}.$$
	\end{theorem}
	\begin{proof}
		Since  $g\in S_{f}(\psi)$, we have $g\prec f$ for a fixed $f\in \mathcal{S}^*(\psi)$ . By letting $r$ tends to $1$ in Lemma~\ref{grth}, we obtain the Koebe-radius $r_{*}=-f_0(-1)$. Therefore
		$\mathbb{B}(0,r_{*}) \subset f(\mathbb{D})$, which implies $  r_{*}\leq d(0,\partial\Omega)=|f(z)|$ for $|z|=1$. Also using \cite[Theorem~1, p.161]{minda94}, we have
		\begin{equation}\label{f-f0}
		\frac{f(z)}{z}\prec \frac{f_0(z)}{z}.
		\end{equation}
		Recall the result \cite[Lemma~1, p.1090]{bhowmik2018}, which reads as:
		let $f$ and $g$ be analytic in $\mathbb{D}$ with $g\prec f,$ where
		$f(z)=\sum_{n=0}^{\infty}a_n z^n$ and $ g(z)= \sum_{k=0}^{\infty}b_k z^k.$
		Then
		$\sum_{k=0}^{\infty}|b_k|r^k \leq \sum_{n=0}^{\infty}|a_n|r^n$ for $ |z|=r\leq1/3.$
		Now using the result for  $g\prec f$ and  \eqref{f-f0}, we have
		\begin{equation*}
		\sum_{k=1}^{\infty}|b_k|r^k \leq	r+\sum_{n=2}^{\infty}|a_n|r^n \leq \hat{f}_0(r)\quad\text{for}\; |z|=r\leq1/3.
		\end{equation*}
		Finally, to establish the inequality \eqref{compare}, it is enough to show $\hat{f}_0(r) \leq r_{*}.$
		But this holds whenever $r\leq r_0$, where $r_0$ is the least positive root of the equation $\hat{f}_0(r)=r_{*}.$
		The existence of the root $r_0$ is ensured by the relations $\hat{f}_0(1)\geq |f_0(1)|\geq r_{*}$ and $\hat{f}_0(0)<r_{*}$. Thus, if $r_b=\min\{r_0, 1/3 \}$ then $\sum_{k=1}^{\infty}|b_k|r^k \leq d(0, \partial{\Omega}) $ holds.
		Hence the result. \qed
	\end{proof}
	
	\begin{remark}
		Let us further assume that the coefficients $B_n$ of $\psi$ are positive. Then the function $f_0(z)=z+\sum_{n=2}^{\infty}t_n z^n$ defined by integral representation \eqref{int-rep} can be written as
		$$f_0(z)=z \exp\left(\sum_{n=1}^{\infty}\frac{B_n}{n}z^n\right),$$
		which implies
		$f_0(r)=\hat{f}_0(r)$ for $|z|=r.$
	\end{remark}	
	
	From the proof of Theorem~\ref{bohr}, we have the following:
	\begin{theorem}
		Let $r_{*}$ be the Koebe-radius for the class $\mathcal{S}^*(\psi),$  $f_0(z)$ be given by the equation~\eqref{int-rep} and $f(z)=z+\sum_{n=2}^{\infty}a_n z^n \in \mathcal{S}^*(\psi)$. Assume  $f_0(z)=z+\sum_{n=2}^{\infty}t_n z^n$ and $\hat{f}_0(r)=r+\sum_{n=2}^{\infty}|t_n|r^n$.
		Then   $\mathcal{S}^*(\psi)$ satisfies the Bohr-phenomenon
		\begin{equation*}
		r+\sum_{n=2}^{\infty}|a_n|r^n \leq d(f(0),\partial\Omega),\quad \text{for}\; |z|=r\leq r_b,
		\end{equation*}
		where $r_b=\min\{r_0, 1/3 \}$, $\Omega=f(\mathbb{D})$ and $r_0$ is the least positive root of the equation
		$$\hat{f}_0(r)=r_{*}.$$
	\end{theorem}

	{\bf Some Applications:}\\
	
	(a). If we choose $\psi(z)=(1+Dz)/(1+Ez)$, $-1\leq E<D\leq1$, then $\mathcal{S}^*(\psi)$ denotes the class of  Janowski starlike functions and we have
	\begin{equation}\label{jb}
	r_{*}=
	\left\{
	\begin{array}
	{lr}
	(1-E)^{\frac{D-E}{E}}, &  E\neq0; \\
	e^{-D},   & E=0.
	\end{array}
	\right.
	\end{equation}
	and
	\begin{equation}\label{jf}
	f_0(z)=
	\left\{
	\begin{array}
	{lr}
	z(1+Ez)^{\frac{D-E}{E}}, &  E\neq0; \\
	z\exp(Dz),   & E=0.
	\end{array}
	\right.
	\end{equation}
	Observe that if $E\neq0$, the $n$-th $(n\geq2)$ coefficients of $f_0(z)$ is given by
	\begin{equation}\label{t_n}
	t_n=\prod_{k=2}^{\infty}\frac{D-(k-1)E}{(k-1)!}.
	\end{equation}
	Thus from Theorem~\ref{bohr}, we have the following result:
	\begin{corollary}\label{janouwski}
		Let $\psi(z)=(1+Dz)/(1+Ez)$, $-1\leq E<D\leq1$. Then $S_{f}(\psi)$ (and $\mathcal{S}^{*}(\psi)$) satisfies the Bohr-phenomenon \eqref{bohrphenomenon} for $|z|=r\leq r_b,$
		where $r_b=\min\{r_0, 1/3 \}$ and $r_0$ is the least positive root of the equation
		$$r+\sum_{n=2}^{\infty}|t_n|r^n-(1-E)^{\tfrac{D-E}{E}}=0,$$
		where $t_n$ is as defined in \eqref{t_n}.
	\end{corollary}
	
	Note that for the Janowski class, sharp coefficients bounds in general are not known. Now as an application of Corollary \ref{janouwski}, we obtain the following result when $t_n>0$:
	\begin{corollary} {\bf{(Bohr-radius with Janowski class)} }\label{jrslt}
		Let $\psi(z)=(1+Dz)/(1+Ez)$, $-1\leq E<D \leq1$.
		\begin{itemize}
			\item [$(i)$]
			If $E=0$ and $D\geq \frac{3}{4}\log{3} $. Then $S_{f}(\psi)$ (and $\mathcal{S}^{*}(\psi)$) satisfies the Bohr-phenomenon \eqref{bohrphenomenon}
			for $|z|=r\leq r_0,$
			where $r_0$ is the only real root of the equation
			\begin{equation}\label{eb0}
			1-re^{D(1+r)}=0.
			\end{equation}
			
			\item [$(ii)$] If $E\neq0$ and further satisfies
			\begin{equation}\label{DE}
			3(1-E)^{\frac{D-E}{E}} \leq (1+E/3)^{\frac{D-E}{E}}.
			\end{equation}
			Then $S_{f}(\psi)$ (and $\mathcal{S}^{*}(\psi)$) satisfies the Bohr-phenomenon
			\eqref{bohrphenomenon} for $|z|=r\leq r_0,
			$
			where $r_0$ is the only real root of the equation
			\begin{equation}\label{eb}
			(1-E)^{\frac{D-E}{E}}-r(1+Er)^{\frac{D-E}{E}}=0.
			\end{equation}
		\end{itemize}
		The result is sharp for the function $f_0$ defined in \eqref{jf}.
	\end{corollary}
	\begin{proof}
		{{(i):}} Since $E=0,$ we have $r_*=e^{-D}$. Moreover $\hat{f}_0(r)=f_0(r)= r\exp(Dr)$. Now we need to show
		\begin{equation}\label{B0S}
		r\exp(Dr)\leq e^{-D}
		\end{equation}
		or equivalently $T(r):=1-r e^{D(1+r)}\geq0$
		holds for $r\leq r_0$. Which obviously holds for $\frac{3}{4}\log{3} \leq D \leq1$. Since $d(f_0(0),\partial f_0(\mathbb{D}))= r_*$, therefore we see from inequality \eqref{B0S} that Bohr-radius is sharp for the function $f_0$ given by \eqref{jf}.\\
		
		{{(ii):}} Proceeding as in case~(i), it is sufficient to show the inequality
		\begin{equation}\label{ABS}
		r(1+Er)^{\frac{D-E}{E}}\leq (1-E)^{\frac{D-E}{E}}
		\end{equation}
		or equivalently $g(r):=(1-E)^{\frac{D-E}{E}}-r(1+Er)^{\frac{D-E}{E}} \geq0$
		holds for $r\leq r_0$. Which obviously follows whenever $D$ and $E$ satisfies \eqref{DE}. In view of the inequality~\eqref{ABS}, the sharp Bohr-radius is achieved for the function $f_0$ given by \eqref{jf}. \qed
	\end{proof}
	
	\begin{remark}{\bf{(Bohr-radius with starlike functions of order $\alpha$ )} }
		Let $\psi(z):=(1+(1-2\alpha)z)/(1-z),$ where  $0\leq \alpha<1$. We see $\mathcal{S}^*(\psi) :=\mathcal{S}^*(\alpha)$ and for this class, we have
		$$r_{*}=\frac{1}{2^{2(1-\alpha)}} \quad \text{and} \quad f_0(z)=\frac{z}{(1-z)^{2(1-\alpha)}}.$$
		Observe that here $\hat{f}_0(r)=f_0(r)$. Now as an application of Corrollary \ref{jrslt}, we obtain the result due to Bhowmik et al. \cite{bhowmik2018}, namely: {\it If $0\leq\alpha\leq1/2$. Then $S_{f}(\psi)$ satisfies the Bohr-phenomenon $\sum_{k=1}^{\infty}|b_k|r^k \leq d(f(0),\partial\Omega),\; \text{for}\; |z|=r\leq r_b,$
			where $r_b=\min\{r_0, 1/3 \}=r_0$ and $r_0$ is the only real root of the equation
			$ {(1-r)^{2(1-\alpha)}}/r=2^{2(1-\alpha)}.$
			The result is sharp.}
	\end{remark}
	
	Now form the above remark, in particular, we have:
	\begin{corollary}
		If $0\leq\alpha\leq1/2$. Then the class $\mathcal{S}^*((1+(1-2\alpha)z)/(1-z))$ satisfies the Bohr-phenomenon  \eqref{bohrphenomenon} for $|z|=r\leq r_0$, where $r_0$ is the only real root of the equation
		$$ {(1-r)^{2(1-\alpha)}}/r=2^{2(1-\alpha)}.$$
		The result is sharp. In particular, the Bohr radius for the class $\mathcal{S}^{*}$ is $3-2\sqrt{2}\approx 0.1713$.
	\end{corollary}
	
	(b). If we choose $\psi(z)=\sqrt{1+z}$, then $\mathcal{S}^{*}(\psi):=\mathcal{SL}^{*}$, the class of lemniscate starlike functions and for this class we have:
	\begin{equation}\label{lemiscate-conj}
	f_0(z)=\frac{4z \exp(2\sqrt{1+z}-2)}{(1+\sqrt{1+z})^2} \quad \text{and}\quad  r_{*}=-f_0(-1)\approx 0.541341.
	\end{equation}
	Also in this case $\hat{f}_0(r)=f_0(r)$ and therefore, we obtain the following corollary:
	\begin{corollary}
		The class $S_{f}(\psi)$ (and $\mathcal{SL}^{*}$), where $\psi(z)=\sqrt{1+z}$ satisfies the Bohr-phenomenon  \eqref{bohrphenomenon} for $|z|=r\leq 1/3.$
	\end{corollary}
	
	(c). If we consider $\psi(z)=1+ze^z$, then $\mathcal{S}^{*}(\psi):=\mathcal{S}^{*}_{\wp}$, the class of cardioid starlike functions introduced in \cite{Kumar-cardioid} and for this class, we have:
	\begin{equation}\label{cardioid-conj}
	f_0(z)=z\exp(e^z-1) \quad \text{and}\quad r_{*}=-f_0(-1)\approx0.531464.
	\end{equation}
	Here we can also see that $\hat{f}_0(r)=f_0(r)$ and we obtain the following corollary:
	\begin{corollary}\label{cardioid}
		The class $S_{f}(\psi)$ (and $\mathcal{S}^{*}_{\wp}$), where $\psi(z)=1+ze^z$ satisfies the Bohr-phenomenon \eqref{bohrphenomenon} for $|z|=r\leq 1/3.$
	\end{corollary}
	
	Ali et al.~\cite{jain2019} also showed that the coefficient bound of a function in  a class have a role in the estimation of the Bohr-radius. Observed that for each $f\in \mathcal{S}^{*}(\psi),$ the class $S_{f}(\psi)$ satisfies the Bohr-phenomenon for $r\leq \min(1/3, r_0)$, where $r_0$ is the least positive root of $\hat{f}_0(r)-r_{*}=0$. Since $\mathcal{S}^*(\psi)\subset \bigcup_{f\in \mathcal{S}^*(\psi)} S_{f}(\psi) $, therefore the Bohr-radius for the class $\mathcal{S}^{*}(\psi)$ is $r\geq \min(1/3, r_0).$ In Corollary~\ref{cardioid}, we find  $r_0\approx0.349681$ (an upper bound for Bohr radius), which is almost close to $1/3\approx0.33333$ and is the unique root of $f_0(r)-r_{*}=0$.
	Moreover, the bound for the coefficients of the functions belonging to $\mathcal{S}^{*}_{\wp}$ and $\mathcal{SL}^{*}$ have been conjectured \cite{Kumar-cardioid,sokolconj-2009} with the extremals given in \eqref{cardioid-conj} and \eqref{lemiscate-conj} respectively. Thus by using Theorem~\ref{bohr} and the approach dealt in \cite{jain2019} (assuming that conjectures are true), we propose the following conjectures:	
\begin{conj}
 The Bohr-radius for the class $\mathcal{S}^{*}_{\wp}$ is $r_0\approx0.349681$ which is the unique root in $(0,1)$ of the equation $$re^{e^r}=e^{1/e}.$$
\end{conj}
\begin{conj}
 The Bohr-radius for the class $\mathcal{SL}^{*}$ is $r_0\approx0.439229$, which the unique root in $(0,1)$ of the equation
	$$e^2 r \exp(2\sqrt{1+r}-2)=(1+\sqrt{1+r})^2.$$
\end{conj}	
\section*{Conflict of interest}
	The authors declare that they have no conflict of interest.

\end{document}